\documentclass[11pt]{article}

\usepackage{amsfonts,amsmath,amssymb,amsthm,color,geometry,cite,graphicx}

\usepackage[english]{babel}
\usepackage [normalem]{ulem}
\usepackage{enumerate}

\newtheorem{theorem}{Theorem}[section]
\newtheorem{proposition}[theorem]{Proposition}

\newtheorem{lemma}[theorem]{Lemma}
\newtheorem{corollary}[theorem]{Corollary}

\numberwithin{equation}{section}
\numberwithin{theorem}{section}

\newcommand{\EE}{{\mathcal E}}

\newcommand{\eps}{\varepsilon}
\newcommand{\E}{\mathcal{E}}

\newcommand{\TT}{\mathbb{T}}
\newcommand{\R}{\mathbb{R}}

\title{Emergence of non-trivial minimizers for the three-dimensional
  Ohta-Kawasaki energy}

\author{Hans Kn\"upfer \footnote{Institut f\"ur Angewandte Mathematik, 
Universit\"at Heidelberg, INF 294, 69120 Heidelberg, Germany}
\and Cyrill B. Muratov \footnote{Department of Mathematical
    Sciences, New Jersey Institute of Technology, Newark, NJ 07102,
    USA} 
\and Matteo Novaga \footnote{Dipartimento di Matematica,
    Universit\`a di Pisa, Largo Bruno Pontecorvo 5, 56127 Pisa,
    Italy}}

\date{\today}

\begin{document}

\maketitle

\begin{abstract}  
  This paper is concerned with the diffuse interface Ohta-Kawasaki
  energy in three space dimensions, in a periodic setting, in the
  parameter regime corresponding to the onset of non-trivial
  minimizers. We identify the scaling in which a sharp transition from
  asymptotically trivial to non-trivial minimizers takes place as the
  small parameter characterizing the width of the interfaces between
  the two phases goes to zero, while the volume fraction of the
  minority phases vanishes at an appropriate rate. The value of the
  threshold is shown to be related to the optimal binding energy
  solution of Gamow's liquid drop model of the atomic nucleus. Beyond
  the threshold the average volume fraction of the minority phase is
  demonstrated to grow linearly with the distance to the threshold. In
  addition to these results, we establish a number of properties of
  the minimizers of the sharp interface screened Ohta-Kawasaki energy
  in the considered parameter regime. We also establish rather tight
  upper and lower bounds on the value of the transition threshold.
\end{abstract}

\tableofcontents

\section{Introduction and main results}

The Ohta-Kawasaki energy is a prototypical energy functional in the
studies of spatially modulated phases that appear as a result of the
competition of short-range attractive and long-range repulsive forces
in physical systems of very different nature. Although it was
originally introduced in the context of microphase separation in
diblock copolymer melts \cite{ohta86}, Ohta-Kawasaki energy is
relevant to a wide range of both soft and hard condensed matter
systems (for a discussion of the specific physical systems see, e.g.,
\cite{m:pre02} and references therein), as well as to dense nuclear
matter at the other extreme of energy and spatial scales
\cite{lattimer85,maruyama05}. From the mathematical point of view,
Ohta-Kawasaki energy functional, together with the closely related
Thomas-Fermi-Dirac-von Weizs\"acker energy
\cite{heisenberg34,weizsacker35,lieb81,lebris05}, serve as a paradigm
for energy-driven pattern forming systems with competing interactions
\cite{cmt:nams17}, which is why the associated variational problem has
received an increasing amount of attention in recent years
\cite{choksi01,choksi09,spadaro09,
  m:cmp10,choksi11,gms:arma13,gms:arma14,lu14}.

In the macroscopic setting, one considers the Ohta-Kawasaki energy
functional defined on a sufficiently large box with periodic boundary
conditions, i.e., for $u \in H^1(\TT_\ell)$ one sets
\begin{align}\label{EEOK}
  \E_\eps(u) := \int_{\TT_\ell} \left(\frac{\eps^2}{2} |\nabla u|^2 + 
  \frac14 (1 - u^2)^2 + \frac12 (u - \bar u_\eps) (-\Delta)^{-1} (u -
  \bar u_\eps) \right) dx ,
\end{align}
where $\TT_\ell$ is the flat $d$-dimensional torus with sidelength
$\ell > 0$, $\eps > 0$ is the parameter characterizing interfacial
thickness and assumed to be sufficiently small, and
$\bar u_\eps \in (-1,1)$ is the parameter equal to the constant
background charge density. In the sequel will investigate the limit
$\eps \to 0$, assuming that $u_\eps$ depends on $\eps$ suitably.  Of
course, the physical dimension of the underlying spatial domain is
$d = 3$. The definition in \eqref{EEOK} needs to be supplemented with
the ``electroneutrality'' constraint in order to make sense of the
last term in \eqref{EEOK}:
\begin{align}
  \label{neutr}
  {1 \over \ell^d} \int_{\TT_\ell} u \, dx = \bar u_\eps.
\end{align}
One then wishes to characterize global energy minimizers of the energy
in \eqref{EEeps} for all $\ell$ sufficiently large. These global
energy minimizers are expected to determine the ground states of the
corresponding physical system in a macroscopically large sample.

It is widely believed that as the value of $\ell$ is increased with
all other parameters fixed, the global energy minimizer of
$\mathcal E_\eps$ should be either constant or spatially periodic,
with period approaching a constant independent of $\ell$ as
$\ell \to \infty$. Proving such a {\em crystallization} result would
be one of the main challenges in the theory of energy-driven pattern
formation and is currently out of reach (for a recent review, see
\cite{blanc15}), except for the case $d = 1$, $\bar u_\eps \in (-1,1)$
fixed and $\eps > 0$ sufficiently small \cite{muller93,ren03,chen06}
(for a very recent result in that direction in higher dimensions, see
\cite{daneri18a}). On the other hand, it is known that for
$\bar u_\eps \in (-1,1)$ fixed, global energy minimizers are not
constant as soon as $\eps \ll 1$ and $\ell \gtrsim 1$
\cite{choksi01,m:cmp10}. This is in contrast with the case
$\bar u_\eps \not \in (-1,1)$, for which by direct inspection
$u = \bar u_\eps$ is the unique global minimizer of the energy. Thus,
a transition from the trivial minimizer $u = \bar u_\eps$ to a
non-trivial, spatially non-uniform minimizer of $\mathcal E_\eps$ must
occur for $\eps \ll 1$ and $\ell \gtrsim 1$ fixed as the value of
$\bar u_\eps$ increases from $\bar u_\eps = -1$ towards
$\bar u_\eps = 0$ (in view of the symmetry exhibited by the energy
when changing $u \to -u$, it is sufficient to consider only the case
$\bar u_\eps \leq 0$). In fact, for $d \geq 2$ non-trivial minimizers
emerge at some $\bar u_\eps = -1 + \delta$ with
$\eps \lesssim \delta \lesssim \eps^{2/3} |\ln \eps|^{1/3}$
\cite{choksi09,m:cmp10}, while for $d = 1$ they emerge for some
$\eps \lesssim \delta \lesssim \eps^{1/2}$ \cite{choksi09,m:pre02}.
The nature of the transition towards non-trivial minimizers is quite
delicate and at present not well understood.

In the absence of general results for non-trivial minimizers of
$\mathcal E_\eps$ for $\eps \lesssim 1$, $\bar u_\eps \in (-1,1)$ and
$\ell \gg 1$ in $d \geq 2$, one can consider different asymptotic
regimes that admit further analytical characterization. One such
regime was analyzed in \cite{gms:arma13}, where the behavior of the
minimizers of $\mathcal E_\eps$ was studied in the limit $\eps \to 0$
for $\bar u_\eps = -1 + \lambda \eps^{2/3} |\ln \eps|^{1/3}$, with
$\lambda > 0$ and $\ell > 0$ fixed, in the case $d = 2$.  In this
regime, non-trivial minimizers are expected to consist of well
separated ``droplets'' of the minority phase, i.e., regions where
$u \simeq +1$ surrounded by the sea of the majority phase where
$u \simeq -1$, separated by narrow domain walls of thickness
$\sim \eps$. It was found that there exists an explicit critical value
of $\lambda = \lambda_c > 0$ such that the minimizers of
$\mathcal E_\eps$ are non-trivial for all $\lambda > \lambda_c$, while
for $\lambda \leq \lambda_c$ the minimizers are ``asymptotically
trivial'', namely, that the energy of the minimizers converges to that
of $u = \bar u_\eps$, and the minimizer converges to $u = \bar u_\eps$
in a certain sense as $\eps \to 0$. Moreover, the threshold value
$\lambda_c$ corresponding to the onset of non-trivial minimizers was
found to be independent of $\ell$, suggesting that the transition
should persist to the macroscopic limit $\ell \to \infty$ with
$\eps \ll 1$ and $\bar u_\eps$ fixed (i.e., when commuting the order
of the $\eps \to 0$ and $\ell \to \infty$ limits). The obtained
non-trivial minimizers exhibit a kind of a homogenization limit, with
mass distributing uniformly on average throughout the
domain. Furthermore, by performing a two-scale expansion of the
energy, one can make more precise conclusions about the detailed
properties of the minimizers and, in particular, formulate a
variational problem in the whole space that determines the placement
of the connected components of the minimizers in terms of the
so-called renormalized energy, whose minimizers are conjectured to
concentrate on the vertices of a hexagonal lattice \cite{gms:arma14}.

Here, we would like to understand how the transition to non-trivial
minimizers happens when $\eps \to 0$ and $\ell \gtrsim 1$ in the
physical three-dimensional case. Therefore, from now on we fix $d = 3$
throughout the rest of the paper. Once again, in this regime the
minimizers are expected to exhibit a two-phase character, with the
minority phase occupying a small fraction of space. To this end, we
define
\begin{align}\label{def-ueps}
  \bar u_\eps:=-1+\lambda \eps^{2/3}\,,
\end{align}
where $\lambda > 0$ is fixed. Our main result is the following
theorem. 

\begin{theorem}
  \label{t:main}
  Let $\ell > 0$ and $\lambda > 0$, and let $\mathcal E_\eps$ be
  defined in \eqref{EEOK} with $\bar u_\eps$ given by
  \eqref{def-ueps}. Then, there exists a universal constant
  $\lambda_c > 0$ such that if $u_\eps$ is a minimizer of
  $\mathcal E_\eps(u)$ among all $u \in H^1(\TT_\ell)$ satisfying
  \eqref{neutr}, and $\mu_\eps \in \mathcal M^+(\TT_\ell)$ is such
  that
  $d \mu_\eps(x) = \frac12 \eps^{-2/3} \, (1 + \mathrm{sgn} \,
  u_\eps(x)) \, dx$, we have as $\eps \to 0$:
  \begin{enumerate}[(i)]
  \item $\mu_\eps \rightharpoonup 0$ in $\mathcal M(\TT_\ell)$ if
    $\lambda \leq \lambda_c$.
  \item $\mu_\eps \rightharpoonup \bar\mu$ in $\mathcal M(\TT_\ell)$,
    where $d \bar\mu = \frac12 (\lambda - \lambda_c) dx$, if
    $\lambda > \lambda_c$.
    
  \item
    $\eps^{-4/3} \mathcal E_\eps(u_\eps) \to \min\{\lambda^2 \ell^3,
    \lambda_c (2 \lambda - \lambda_c) \ell^3 \}$.
  \end{enumerate}
\end{theorem}

Thus, the onset of non-trivial minimizers in three space dimensions
occurs sooner in terms of $0 < 1 + \bar u_\eps \ll 1$ than the
corresponding transition in two dimensions. In particular, cylindrical
morphologies obtained by trivially extending the two-dimensional
minimizers into the third dimension are no longer global energy
minimizers. One would, therefore, expect that the emergent non-trivial
minimizers consist of a collection of well separated small droplets of
the minority phase surrounded by the sea of the majority
phase. Furthermore, the size and the distance between the droplets
would scale differently (cf. \cite{kmn:cmp16}) from those in two
dimensions, and in contrast to the latter \cite{gms:arma13,gms:arma14}
we can no longer conclude that the droplets are nearly
spherical. Still, we are able to express the shape and size of the
individual droplets in terms of the non-local isoperimetric problem in
the whole space that goes back to Gamow \cite{gamow30,cmt:nams17} (for
details, see the following sections). In particular, this allows us to
obtain a quantitative estimate for the threshold $\lambda_c$, using
balls as competitors in the whole space problem and a recent
quantitative non-existence result for the Gamow's liquid drop model
\cite{frank16}.

\begin{theorem}
  \label{t:f*}
  With the notation of Theorem \ref{t:main}, we have
  \begin{align}
    \label{lamcB}
    {3 \over 4 \sqrt[3]{2}} \leq \lambda_c \leq {3 \over 2 \sqrt[3]{5}}.
  \end{align}
\end{theorem}

Note that numerically the bound in Theorem \ref{lamcB} appears to be
fairly tight: $0.5952 < \lambda_c < 0.8773$, with the lower bound to
within $33\%$ of the value of the upper bound. Note that if the
conjecture that the minimizers of Gamow's liquid drop model are balls
is true, then the upper bound in Theorem \ref{lamcB} should in fact
yield equality.

Our proof relies on our previous results obtained for the
three-dimensional sharp interface version of the Ohta-Kawasaki energy
\cite{kmn:cmp16}. Together with the approach from \cite[Section
4]{m:cmp10}, the result in Theorem \ref{t:main} is obtained along the
lines of the arguments in \cite{gms:arma13}, suitably adapted from the
two-dimensional to the three-dimensional case. Note, however, that the
results in \cite{kmn:cmp16} cannot be directly combined with those of
\cite{m:cmp10}, since the sharp interface energy studied in
\cite{kmn:cmp16} does not include the effect of {\em charge
  screening}. In fact, there is no transition from trivial to
non-trivial minimizers in the unscreened sharp interface
energy. Therefore, as a first step towards the proof one needs to
adapt the results of \cite{kmn:cmp16} to the case of screened sharp
interface energy and obtain an asymptotic characterization of its
minimizers as $\eps \to 0$.  

As in \cite{kmn:cmp16}, we separate the non-local energy into the
near-field and far-field contributions, with screening appearing
explicitly in the latter. At the same time, the self-interaction
energy of the droplets turns out to be still well approximated by that
of Gamow's liquid drop model. Combining the far-field with the
near-field contributions to the energy then allows to establish a
$\Gamma$-convergence result for the screened sharp interface energy to
an energy functional which is quadratic in the limit charge density,
with the notion of convergence being the weak convergence of
measures. Along the way, we establish uniform estimates for the
connected components of the minimizers similar to those in
\cite{kmn:cmp16}, which, in turn, allows to characterize non-existence
of nontrivial minimizers of the screened sharp interface energy below
the threshold for all $\eps$ sufficiently small.

Once the $\Gamma$-convergence result is established for the screened
sharp interface energy, we proceed as in \cite{gms:arma13} by
introducing a piecewise-constant charge density associated with the
admissible configurations for the diffuse interface energy that
eliminates the small deviations of the charge density from their
equilibrium values $\pm 1$ for the double-well potential (for a more
detailed explanation of the need of such a step, see the beginning of
Sec. 2.2 in \cite{gms:arma13}). We then adapt the arguments of
\cite[Section 6]{gms:arma13} to obtain the corresponding
$\Gamma$-convergence result for the diffuse interface energy to the
same quadratic functional in the limit charge density as for the
screened sharp interface energy. Finally, explicitly minimizing the
limit energy we obtain the main result of our paper contained in
Theorem \ref{t:main}. Furthermore, we relate the value of the
threshold $\lambda_c$ with the optimal energy per unit mass for
Gamow's liquid drop model. In addition, we use recent results in
\cite{frank15,frank16} characterizing the minimizers of the latter
problem to obtain sharp quantitative bounds on the value of the
threshold.

To summarize, our paper provides an extension of various recent
results for the diffuse interface Ohta-Kawasaki energy to the case of
a macroscopic three-dimensional domain, establishing a sharp
transition from trivial to nontrivial minimizers in the asymptotic
limit of vanishingly thin interfaces. Most of the techniques used in
our proofs are adaptations of those that appeared in the earlier
studies of this problem in different settings. The main novelty of our
results, however, is the way these arguments are combined to yield a
non-trivial scaling for the transition to nontrivial minimizers and
the limit energy functional for the three-dimensional Ohta-Kawasaki
energy. To our knowledge, this is the first sharp asymptotic result
for this energy in the regime of strong compositional asymmetry and
large number of droplets (for the case of finitely many droplets, see
\cite{choksi11}). We note that the present lack of knowledge about the
minimizers of Gamow's liquid drop model prevents us to go to the next
order in a two-scale $\Gamma$-expansion to describe local interactions
of droplets via a ``renormalized energy'' \cite{rougerie16}. In
particular, it is not known at present whether the minimizer per unit
mass exists only for a unique value of the mass (this would be true if
minimizers were balls). Thus, further insights into the solution of
Gamow's model would be needed to carry out the programme realized for
the two-dimensional Ohta-Kawasaki energy in \cite{gms:arma14}.

Our paper is organized as follows. In Sec.~\ref{sec:setting}, we
introduce the different energies appearing in our study and state a
number of results related to each of the associated variational
problems. Also in this section, we prove Theorem \ref{t:f*lb} that
gives a quantitative lower bound for the self-interaction energy per
unit mass for Gamow's liquid drop model. Then, in Sec.~\ref{sec:sharp}
we state the $\Gamma$-convergence result for the sharp interface
energy in Theorem \ref{teogamma}, followed by a proof. Also in
Sec. \ref{sec:sharp}, we provide some further results about the
connected components of minimizers of the screened sharp interface
energy, see Theorem \ref{t:diam} and Corollary \ref{t:triv}.  Finally,
in Sec. \ref{sec:diffuse} we state and prove the corresponding
$\Gamma$-convergence result for the diffuse interface energy, see
Theorem \ref{teogammabis}. The results in Theorems \ref{t:main} and
\ref{t:f*} are then obtained as simple corollaries of the above
theorems.

\section{Setting}
\label{sec:setting}

In this section, we introduce the basic notation used throughout the
rest of the paper, together with the assumptions and some technical
results.

\subsection{The diffuse interface energy} 
\label{subsec:diffuse}

We begin by generalizing the diffuse energy functional in \eqref{EEOK}
to one involving an arbitrary symmetric double-well potential $W(u)$:
\begin{align}\label{EEeps} %
  \E_\eps(u) := \int_{\TT_\ell} \left(\frac{\eps^2}{2} |\nabla u|^2 + 
  W(u) + \frac12 (u - \bar u_\eps) (-\Delta)^{-1} (u - \bar u_\eps)
  \right) dx,
\end{align}
with $W(u)$ satisfying \cite{m:cmp10}:
\begin{enumerate}
\item[(i)] $W \in C^2(\mathbb R)$, $W(u) = W(-u)$, and $W \geq 0$, 

\item [(ii)] $W(+1) = W(-1) = 0$ and $W''(+1) = W''(-1) > 0$,

\item [(iii)] $W''(|u|)$ is monotonically increasing for $|u| \geq 1$,
  $\lim_{|u| \to \infty} W''(u) = +\infty$, and
  $|W'(u)| \leq C (1 + |u|^q)$, for some $C > 0$ and $1 < q <5$.
\end{enumerate}
This energy is clearly well-defined and bounded on the admissible
class
\begin{align}  \label{Aeps} %
  \mathcal A_\eps := \left\{ u \in H^1(\TT_\ell) \ : \ {1 \over
  \ell^3} \int_{\TT_\ell} u \, dx = \bar u_\eps \right\},
\end{align}
with the non-local term interpreted, as usual, with the help of the
Green's function $G_0(x)$ solving
\begin{align}
  \label{G0}
  -\Delta G_0(x) = \delta(x) - \ell^{-3}, \qquad \int_{\TT_\ell}
  G_0(x) \, dx = 0
\end{align}
in $\mathcal D'(\TT_\ell)$. Explicitly, the energy takes the form
\begin{align}
  \label{EEepsG}
  \EE_\eps(u) = \int_{\TT_\ell} \left( {\eps^2 \over 2} |\nabla u|^2 +
  W(u) \right) dx + \frac12 \int_{\TT_\ell} \int_{\TT_\ell}  (u(x) - \bar 
  u_\eps) G_0(x - y) (u(y) - \bar u_\eps) \, dx \, dy,
\end{align}
noting that the last term in the right-hand side is well-defined by
Young's inequality.

Under the assumptions above, every critical point
$u \in \mathcal A_\eps $ of $\EE_\eps$ solves weakly the
Euler-Lagrange equation, which can be written as (see \cite[Section
4]{m:cmp10})
\begin{align}
  \label{ELeps}
  -\eps^2 \Delta u + W'(u) + v = \Lambda,  \qquad
  - \Delta v = u - \bar u_\eps,
\end{align}
where $v \in H^3(\TT_\ell)$ is a zero-average solution of the second
equation in \eqref{ELeps} and $\Lambda \in \R$ is the Lagrange
multiplier satisfying
\begin{align}
  \label{mu}
  \Lambda = {1 \over \ell^3} \int_{\TT_\ell} W'(u) \, dx,
\end{align}
as can be seen by integrating the first equation in \eqref{ELeps} over
$\TT_\ell$. In particular, we have
\begin{align}
  \label{vG0}
  v(x) = \int_{\TT_\ell} G_0(x - y) (u(y) - \bar u_\eps) dy,
\end{align}
and $u, v \in C^\infty(\TT_\ell)$, solving \eqref{ELeps} classically
\cite[Section 4]{m:cmp10}. Also, by the direct method of calculus of
variations, minimizers of $\EE_\eps$ are easily seen to exist for all
choices of the parameters.

\subsection{The sharp interface energy with screening} 

For $\eps \ll 1$, minimizers of $\EE_\eps$ are expected to consist of
functions which take values close to $\pm 1$, except for narrow
transition regions of width of order $\eps$ \cite{m:cmp10}. As usual,
we define the energy of an optimal one-dimensional transition layer
connecting $u = \pm 1$ \cite{modica87}:
\begin{align}
  \label{Wint}
  \sigma := \int_{-1}^{1} \sqrt{2 W(s)} \, ds > 0.
\end{align}
We also define 
\begin{align} \label{kappa} %
  \kappa:= {1 \over \sqrt{W''(1)}} > 0,
\end{align}
characterizing the effect of charge screening appearing in the sharp
interface version of the energy $\EE_\eps$, which we introduce in the
sequel. With some obvious modifications, the results of \cite[Section
4]{m:cmp10} apply to $\EE_\eps$ defined in \eqref{EEeps}, with the
corresponding sharp interface energy $E_\eps$ defined as
\begin{align}\label{Eeps} %
  E_\eps(u) %
  := \frac{\eps \sigma}{2} \int_{\TT_\ell} |\nabla u| \, dx 
  + \frac12 \int_{\TT_\ell}  (u - \bar u_\eps) (-\Delta +
  \kappa^2)^{-1} (u - \bar
  u_\eps) \,  dx,
\end{align}
where $u$ belongs to the admissible class
\begin{align} \label{A} %
  \mathcal A := BV(\TT_\ell; \{-1, 1\}).
\end{align}
Specifically, in the considered scaling regime we have the following
relation between the two energies (see the following sections):
\begin{align}
  \label{EepsEEeps}
  {\displaystyle \min_{u \in \mathcal A_\eps}
  \EE_\eps(u) \over \displaystyle 
  \min_{u \in \mathcal A} E_\eps(u)} \to 1 \qquad \text{as} \quad \eps
  \to 0.
\end{align}
Notice that the neutrality constraint in \eqref{neutr} is no longer
present in the case of the sharp interface energy. 

The energy in \eqref{Eeps} may be rewritten with the help of the
Green's function as
\begin{align}\label{Eeps2} %
  E_\eps(u) %
  := \frac{\eps \sigma}{2} \int_{\TT_\ell} |\nabla u| \, dx 
  + \frac12 \int_{\TT_\ell}  \int_{\TT_\ell}  (u(x) - \bar u_\eps) G(x
  - y) (u(y) - \bar
  u_\eps) \,  dx \, dy,
\end{align}
where $G$ solves
\begin{align} 
  \label{G2} 
  -\Delta G(x) + \kappa^2 G(x) = \delta(x)
  \qquad \text{in} \quad \mathcal D'(\TT_\ell).
\end{align}
Notice that $G$ has an explicit representation
\begin{align}
  \label{Gsum}
  G(x) = {1 \over 4 \pi} \sum_{\mathbf n \in \mathbb Z^3} {e^{-\kappa
  |x - \mathbf n \ell|} \over  |x - \mathbf n \ell|}. 
\end{align}
In particular, we have
\begin{align}
  \label{Gprops}
  G(x) \simeq {1 \over 4 \pi |x|} \qquad |x| \ll 1, 
  \qquad \qquad G(x) \geq c \qquad \forall x \in \TT_\ell,
\end{align}
for some $c > 0$ depending on $\kappa$ and $\ell$.  Also, integrating
\eqref{G2} we get
\begin{align} \label{Gint} %
  \int_{\TT_\ell} G(x) dx = \kappa^{-2}.
\end{align}
The latter allows us to rewrite the energy $E_\eps$ in an equivalent
form in terms of $\chi \in BV(\TT_\ell; \{0,1\})$, where
\begin{align}
  \label{chiu}
  \chi(x) := {1 + u(x) \over 2} \qquad \qquad x \in \TT_\ell,
\end{align}
as
\begin{multline}
  \label{Eepschi}
  E_\eps(u) = {\eps^{4/3} \lambda^2 \ell^3 \over 2 \kappa^2} + \eps
  \sigma \int_{\TT_\ell} |\nabla \chi| \, dx - {2 \eps^{2/3} \lambda
    \over \kappa^2} \int_{\TT_\ell} \chi \, dx + 2 \int_{\TT_\ell}
  \int_{\TT_\ell} G(x - y) \chi(x) \chi(y) \, dx \, dy ,
\end{multline}
where we also used \eqref{neutr}.

We now introduce a version of the energy $E_\eps$ written in terms of
the rescaling
\begin{align}
  \label{chitil}
  \tilde\chi(x) := \chi(\ell x / \ell_\eps) \quad x \in
  \TT_{\ell_\eps}, \qquad \quad \ell_\eps := \left( {4 \over \sigma \eps}
  \right)^{1/3} \ell. 
\end{align}
With this definition we have $\tilde \chi \in \widetilde{\mathcal A}_{\ell_\eps}$,
where
\begin{align}
  \label{Atileps}
  \widetilde{\mathcal A}_{\ell_\eps} := BV(\TT_{\ell_\eps}; \{0, 1\}),
\end{align}
for every $\chi \in \mathcal A$, and
$E_\eps(\chi) = \widetilde E_{\ell_\eps}(\tilde \chi)$, with
\begin{multline}
  \label{Eepstil}
  \widetilde E_{\ell_\eps}(\tilde \chi) := {\eps^{4/3} \lambda^2
    \ell^3 \over 2 \kappa^2} - {\eps^{5/3} \sigma \lambda \over 2
    \kappa^2} \int_{\TT_{\ell_\eps}} \tilde \chi \, dx \\ + \left(
    {\eps^{5/3} \sigma^{5/3} \over 4^{2/3}} \right) \left[
    \int_{\TT_{\ell_\eps}} |\nabla \tilde \chi| \, dx + \frac12
    \int_{\TT_{\ell_\eps}} \tilde \chi \left( -\Delta + 4^{-2/3}
      \eps^{2/3} \sigma^{2/3} \kappa^2 \right)^{-1} \tilde \chi \, dx
  \right].
\end{multline}
Introducing $G_\eps$, which solves
\begin{align}
  \label{Geps}
  -\Delta G_\eps(x) + 4^{-2/3} \kappa^2 \eps^{2/3} \sigma^{2/3}
  G_\eps(x) = \delta(x) \qquad \text{in} \ \TT_{\ell_\eps},
\end{align}
we can then express the energy $\widetilde E_{\ell_\eps}$ as 
\begin{multline}
  \label{Eepstil2}
   \widetilde E_{\ell_\eps}(\tilde \chi) := {\eps^{4/3} \lambda^2
    \ell^3 \over 2 \kappa^2} - {\eps^{5/3} \sigma \lambda \over 2
    \kappa^2} \int_{\TT_{\ell_\eps}} \tilde \chi \, dx \\ + \left(
    {\eps^{5/3} \sigma^{5/3} \over 4^{2/3}} \right) \left[
    \int_{\TT_{\ell_\eps}} |\nabla \tilde \chi| \, dx + \frac12
    \int_{\TT_{\ell_\eps}} G_\eps(x - y) \tilde \chi(x) \tilde \chi(y)
    \, dx \, dy
  \right].
\end{multline}
Note that, as in \eqref{Gsum}, we have the following representation for
$G_\eps$:
\begin{align}
  \label{Gsumeps}
  G_\eps(x) = {1 \over 4 \pi} \sum_{\mathbf n \in \mathbb Z^3} {e^{-
  4^{-1/3} \eps^{1/3} \sigma^{1/3} \kappa | x - \mathbf n \ell_\eps|}
  \over  | x - \mathbf n \ell_\eps|}. 
\end{align}

\subsection{The whole space energy}

As was shown by us in \cite{kmn:cmp16}, in the absence of screening,
i.e., with $\kappa = 0$ and $u \in \mathcal A$ also satisfying
\eqref{neutr}, the asymptotic behavior of the minimizers of $E_\eps$
in \eqref{Eeps} with $\bar u_\eps$ satisfying \eqref{def-ueps} can be
expressed in terms of those for the energy defined on the whole of
$\R^3$:
\begin{align}\label{eninfty}
  \widetilde E_\infty(\tilde \chi) := \int_{\R^3} |\nabla \tilde \chi|
  \, dx + \frac{1}{8 \pi}  \int_{\R^3}  \int_{\R^3} {\tilde \chi (x)
  \tilde \chi (y) \over |x - y|} \, dx \, dy, 
\end{align}
which is well defined in the admissible class
\begin{align}
  \label{Ainfty}
  \widetilde{\mathcal A}_\infty := BV(\R^3; \{0, 1\}).
\end{align}
In particular, the optimal self-energy per unit volume of the minority
phase is
\begin{align}
  \label{f*}
  f^* := \inf_{\tilde \chi \in \widetilde{\mathcal A}_\infty}
  {\widetilde E_\infty(\tilde \chi) \over 
  \int_{\R^3} \tilde \chi \, dx}. 
\end{align}
Note that within the nuclear physics context, this is precisely the
dimensionless form of the celebrated Gamow's liquid drop model of the
atomic nucleus \cite{gamow30} (for a recent mathematical overview, see
\cite{cmt:nams17}). In particular, the value of $f^*$ corresponds to
the energy per nucleon in the tightest bound nucleus.

The relationship between $E_\eps$ and $\widetilde E_\infty$ can be
seen formally by passing to the limit $\eps \to 0$ in \eqref{Eepstil2}
with $\tilde \chi$ taken to be the characteristic function of a fixed
bounded set restricted to $\TT_{\ell_\eps}$. Then we have
\begin{align}
  \left( {4^{2/3} \over \eps^{5/3} \sigma^{5/3}} \right) \left(
  \widetilde E_{\ell_\eps}(\tilde \chi) - {\eps^{4/3} \lambda^2 \ell^3
  \over 2 \kappa^2} \right) \to  -{\lambda f^* \over \lambda_c}
  \int_{\R^3} \tilde \chi \, dx + \widetilde E_\infty(\tilde \chi), 
\end{align}
where $\tilde \chi$ was extended by zero to the whole of $\R^3$, and
we defined 
\begin{align}
  \label{lamc}
  \lambda_c := 2^{-1/3}\sigma^{2/3} \kappa^2 f^*.
\end{align}

The following result was recently established about minimizers of the
problem in the whole space \cite{kmn:cmp16,frank16}.

\begin{theorem}
  \label{t:gamow}
  There exists a bounded, connected open set $F^* \subset \R^3$ with
  smooth boundary such that
  \begin{align}
    f^* =  {\widetilde E_\infty(\tilde \chi_{F^*}) \over 
    \int_{\R^3} \tilde \chi_{F^*} \, dx},
  \end{align}
  where $\chi_{F^*}$ is the characteristic function of the set $F^*$. 
\end{theorem}

It has been conjectured that the minimizer of $\widetilde E_\infty$
with fixed mass is given by a ball whenever such a minimizer exists
\cite{choksi11}. Therefore, taking a ball of radius $R$ as a test
function in \eqref{f*} and optimizing in $R$, one obtains an estimate
\begin{align}
  \label{f*B}
  f^* \leq 3^{5/3}
  \cdot 2^{-2/3} \cdot 5^{-1/3}.  
\end{align}
The above conjecture would imply that the inequality in \eqref{f*B} is
in fact an equality. Proving such a result is a difficult hard
analysis problem that currently appears to be out of
reach. Nevertheless, we can establish a first quantitative lower bound
for the value of $f^*$, using equipartition of energy of $F^*$
established in \cite{frank15} and a quantitative upper bound on
$|F^*|$ obtained in \cite{frank16}. Note that the resulting lower
bound equals about 67\% of the upper bound in \eqref{f*B}. This is one
of the main results of the present paper.

\begin{theorem}
  \label{t:f*lb}
  We have
  \begin{align}
    \label{eq:f*lb}
    f^* \geq {3^{5/3} \over 4}.
  \end{align}
\end{theorem}

\begin{proof}
  Let $F^*$ be a minimizer from Theorem \ref{t:gamow}, and write 
  \begin{align}
    f^* = {P(F^*) + V(F^*) \over |F^*|},
  \end{align}
  where $P(F^*)$ is the perimeter of $F^*$ and $V(F^*)$ is the
  Coulombic self-energy of $F^*$. By the result from \cite{frank15},
  the energy exhibits a kind of equipartition
  \begin{align}
    \label{equipF*}
    V(F^*) = \frac12 P(F^*),
  \end{align}
  which can be easily seen by considering the sets $\lambda F^*$ as
  competitors for $f^*$ and taking advantage of the homogeneity of $P$
  and $V$ with respect to dilations. Thus, we have
  \begin{align}
    f^* = {3 P(F^*) \over 2 |F^*|}.
  \end{align}
  Therefore, applying the isoperimetric inequality yields
  \begin{align}
    f^* \geq \left( {243 \pi \over 2 |F^*|} \right)^{1/3}.
  \end{align}
  The proof is then concluded by recalling the quantitative upper
  bound $|F^*| \leq 32 \pi$ from \cite{frank16}.
\end{proof}

\subsection{The limit energy}

For $\mu \in \mathcal{M}^+(\TT_\ell) \cap H^{-1}(\TT_\ell)$, define
\begin{align}
  \label{E0}
  E_0(\mu) := \frac{\lambda^2 \ell^3}{2\kappa^2} - {2 \over \kappa^2}
  \left(\lambda - \lambda_c \right) \int_{\TT_\ell} d\mu +
  2\int_{\TT_\ell} \int_{\TT_\ell} G(x - y) d \mu(x) d \mu(y).
\end{align}
Note that $\mu \in \mathcal M^+(\TT_\ell) \cap H^{-1}(\TT_\ell)$ implies that
$\mu$ is a non-negative Radon measure that has bounded Coulombic
energy:
\begin{align} \label{dmubdG} \int_{\TT_\ell} \int_{\TT_\ell} G(x - y)
  \, d \mu(x) \, d \mu(y) < \infty.
\end{align}
where $G$ is the screened Coulombic kernel from \eqref{Gsum}. The
converse is also true, i.e., a positive Radon measure with bounded
Coulombic energy defines a bounded linear functional on
$H^1(\TT_\ell)$. This fact follows from the following lemma, whose
proof is a straightforward adaptation of the proof of \cite[ Lemma
3.2]{gms:arma13} in two dimensions. In particular, it allows to extend
the definition of $E_0$ to arbitrary positive Radon measures on
$\TT_\ell$, with $E_0(\mu) < +\infty$ if and only if
$\mu \in H^{-1}(\TT_\ell)$.

\begin{lemma}   \label{lem-Gest} %
  Let $\mu \in \mathcal M^+(\TT_\ell)$ and let \eqref{dmubdG} hold. Then
  \begin{itemize}
  \item[(i)] $\mu \in H^{-1}(\TT_\ell)$, in the sense that it can be
    extended to a bounded linear functional over $H^1(\TT_\ell)$.
  \item[(ii)] If
    \begin{align}
      \label{vGconv}
      v(x) := \int_{\TT_\ell} G(x - y) \, d \mu(y),
    \end{align}
    then $v \in H^1(\TT_\ell)$. Furthermore, $v$ solves
    \begin{align}
      \label{vmu}
      -\Delta v + \kappa^2 v = \mu,
    \end{align}
    weakly in $H^1(\TT_\ell)$, and
    \begin{align}
      \label{dvmu}
      \nabla v(x) = \int_{\TT_\ell} \nabla G(x - y) \, d \mu(y),
    \end{align}
    in the sense of distributions. 
  \item[(iii)] If $v$ is as in (ii), we have $\kappa^2 \int_{\TT_\ell} v \,
    dx = \int_{\TT_\ell} d \mu$ and
    \begin{align}
      \label{Gmumu}
      \int_{\TT_\ell} \int_{\TT_\ell} G(x - y) \, d \mu(x) \, d \mu(y)
      = \int_{\TT_\ell} \left( |\nabla v|^2 + \kappa^2 v^2 \right) dx.
    \end{align}
  \end{itemize}
\end{lemma}

According to Lemma \ref{lem-Gest}, the energy $E_0$ may be
equivalently rewritten in terms of the associated potential $v$ in
\eqref{vGconv} as
\begin{align}
  \label{E0v}
  E_0(\mu) = \frac{\lambda^2 \ell^3}{2\kappa^2} - 2 (\lambda -
  \lambda_c) \int_{\TT_\ell} v 
  \, dx + 2 \int_{\TT_\ell} \left( |\nabla v|^2 + \kappa^2 v^2 \right)
  dx, 
\end{align}
and minimizing $E_0(\mu)$ over
$\mu \in \mathcal{M}^+(\TT_\ell) \cap H^{-1}(\TT_\ell)$ is the same as
minimizing the right-hand side of \eqref{E0v} with respect to all
$v \in H^1(\TT_\ell)$ such that $v \geq 0$ in $\TT_\ell$ and
$-\Delta v + \kappa^2 v \in \mathcal M^+(\TT_\ell)$. By inspection,
the latter is minimized by $v = \bar v$, where
\begin{align}
  \label{vbar}
  \bar v =
  \begin{cases}
    0, & \lambda \leq \lambda_c, \\
    {1 \over 2 \kappa^2} (\lambda - \lambda_c), & \lambda > \lambda_c.
  \end{cases}
\end{align}

In terms of the measures, we can state this result as follows:

\begin{proposition}
  \label{p:minE0}
  The energy $E_0(\mu)$ is minimized by a unique measure $\bar \mu$
  among all $\mu \in \mathcal{M}^+(\TT_\ell) \cap H^{-1}(\TT_\ell)$,
  with $\bar \mu = 0$ for all $\lambda \leq \lambda_c$, and
  $d \bar \mu = \frac12 (\lambda - \lambda_c) dx$ for
  $\lambda > \lambda_c$, respectively. Moreover, we have
  \begin{align}
    \label{E0min}
    E_0(\bar \mu) =
    \begin{cases}
      \frac{\lambda^2 \ell^3}{2\kappa^2}, & \lambda \leq \lambda_c,
      \\
     {\lambda_c(2 \lambda - \lambda_c) \ell^3 \over 2 \kappa^2} , &
     \lambda > \lambda_c,
   \end{cases}
  \end{align}
  and \eqref{vmu} is solved by $v(x) = \bar v$.
\end{proposition}

\section{Sharp interface energy $E_\eps$} \label{sec:sharp} %

In this section, we consider the sharp-interface functional $E_\eps$
defined in \eqref{Eeps} in the limit $\eps \to 0$ with $\bar u_\eps$
given by \eqref{def-ueps} and positive $\sigma, \lambda, \kappa, \ell$
fixed. For a given $u_\eps \in \mathcal A$, we introduce a measure
$\mu_\eps$ that is continuous with respect to the Lebesgue measure on
$\TT_\ell$ and whose density is an appropriately rescaled
characteristic function of the minority phase:
\begin{align} \label{def-mueps} %
  d \mu_\eps(x) := \dfrac12 \eps^{-2/3} (1 + u_\eps(x)) dx.
\end{align}
Note that by definition the measure $\mu_\eps$ is non-negative. We
also introduce the potential $v_\eps$ via
\begin{align} \label{def-veps} %
  -\Delta v_\eps + \kappa^2 v_\eps = \mu_\eps \qquad \text{in} \quad
  \TT_\ell. 
\end{align}
Our first result establishes compactness of sequences with bounded
energy after a suitable rescaling.


\begin{theorem}[Equicoercivity]
  \label{thm-Ecompact} %
  Let $(u_\eps) \in \mathcal A$ be such that
  \begin{align}\label{ZO.1} \limsup_{\varepsilon \to 0}
    \eps^{-4/3} E_\eps(u_\eps) < +\infty,
  \end{align}
  and let $\mu_\eps$ and $v_\eps$ be defined in \eqref{def-mueps} and
  \eqref{def-veps}, respectively.  Then, up to extraction of a
  subsequence, we have
\begin{align} 
  \mu_\eps \rightharpoonup \mu \textrm{ in } \mathcal M(\TT_\ell), \qquad
  v_\eps \rightharpoonup v \textrm{ in } H^1(\TT_\ell), 
\end{align}
as $\eps \to 0$, for some
$\mu \in \mathcal{M}^+(\TT_\ell) \cap H^{-1}(\TT_\ell)$ and
$v \in H^1(\TT_\ell)$ satisfying
\begin{align}
  \label{PDE}
  -\Delta v + \kappa^2 v = \mu
  \qquad \text{in} \quad \TT_\ell.
\end{align}
\end{theorem}
\begin{proof}
  Inserting \eqref{def-mueps} into \eqref{Eeps} and dropping the
  perimeter term, following the argument of \cite{m:cmp10} we arrive
  at (see also \eqref{Eepschi})
  \begin{align} \label{EEOm} %
    E_\eps(u_\eps) & \geq \frac{\lambda^2 \ell^3}{2\kappa^2} - \frac
                     {2\lambda}{\kappa^2} \int_{\TT_\ell} d
                     \mu_\eps(x) + 2 \int_{\TT_\ell}\int_{\TT_\ell}
                     G(x-y) d\mu_\eps(x) d\mu_\eps(y),
\end{align}
where we used \eqref{Gint} and \eqref{def-ueps} and took into account
the translational invariance of the problem in $\TT_\ell$. By
\eqref{Gprops} we get
\begin{align} 
  E_\eps(u_\eps) %
  &\geq - \frac {2\lambda}{\kappa^2} \mu_\eps(\TT_\ell) + 2c
    \mu_\eps^2(\TT_\ell),
\end{align}
where we again recall that $\mu_\eps$ is nonnegative by definition. It
then follows that
\begin{align} 
  \mu_\eps(\TT_\ell) < C,
\end{align}
for some constant $C > 0$ independent of $\eps$, which implies that
$\mu_\eps \rightharpoonup \mu$ for a subsequence. The above
considerations together with Lemma \ref{lem-Gest}(iii) and
\eqref{ZO.1} show that
\begin{align}
  \int_{\TT_\ell}
  \left( |\nabla v_\eps|^2 + \kappa^2 v_\eps^2 \right) dx =
  \int_{\TT_\ell}\int_{\TT_\ell} G(x-y) d\mu_\eps(x) d\mu_\eps(y) < C,
\end{align}
and upon extraction of a further subsequence we get
$v_\eps \rightharpoonup v$ in $H^1(\TT_\ell)$. Finally, \eqref{PDE}
follows by passing to the limit in \eqref{def-veps}.
\end{proof}


We now proceed to the main result of this section which establishes
the $\Gamma$-limit of the screened sharp interface energy, similar to
its two-dimensional analog in \cite[Theorem 1]{gms:arma13}.

\begin{theorem}[$\Gamma$-convergence of $E_\eps$] \label{teogamma} %
  \noindent As $\eps \to 0$ we have
  \begin{align}
    \label{EepsE0}
    \eps^{-4/3} E_\eps \stackrel{\Gamma}\rightharpoonup E_0,    
  \end{align}
  with respect to the weak convergence of measures.  More precisely,
  we have
  \begin{itemize} 
  \item[i)] Lower bound: Suppose that $(u_\eps) \in \mathcal A$ and
    let $\mu_\eps$ be defined as in \eqref{def-mueps}, and suppose
    that
    \begin{align} 
      \mu_\eps \rightharpoonup \mu \textrm{ in } \mathcal M(\TT_\ell), 
    \end{align}
    as $\eps \to 0$, for some $\mu \in
    \mathcal{M_\ell}^+(\TT_\ell)$. Then
    \begin{align} \label{Eliminf} %
      \liminf_{\varepsilon \to 0} \eps^{-4/3} E_\eps(u_\eps) \geq E_0(\mu).
    \end{align}

  \item[ii)] Upper bound: Given $\mu \in \mathcal{M}^+(\TT_\ell)$,
    there exists $(u_\eps) \in \mathcal A$ such that for the
    corresponding $\mu_\eps$ as in \eqref{def-mueps} we have
\begin{align} 
  \mu_\eps \rightharpoonup \mu \text{ in } \mathcal M(\TT_\ell), 
\end{align}
as $\eps \to 0$, and
\begin{align}
  \label{Elimsup}
  \limsup_{\varepsilon \to 0} \eps^{-4/3} 
  E_\eps(u_\eps) \leq E_0(\mu).  
\end{align}
\end{itemize}
\end{theorem}

\begin{proof}
  Assume first that
  $\mu \in \mathcal M^+(\TT_\ell) \cap H^{-1}(\TT_\ell)$, so that
  $E_0(\mu) < +\infty$. As in the proof of Propositions 5.1 and 5.2 in
  \cite{kmn:cmp16}, we separate the contributions of the near-field
  and far-field interaction, i.e. for $0 < \rho \leq \frac14$ we write
  \begin{align} \label{GHrho} %
    G_\rho(x) = \eta_\rho(x) G(x), \qquad H_\rho(x) := G(x) - G_\rho(x),
  \end{align}
  where $\eta_\rho(x)$ is a smooth cutoff function depending on $|x|$
  which is monotonically increasing from 0 to 1 as $|x|$ goes from 0
  to $\rho$, with $\eta_\rho(x) = 0$ for all $|x| < \frac12 \rho$ and
  $\eta_\rho(x) = 1$ for all $|x| > \rho$.  With the help of
  \eqref{Eepschi}, for any $u_\eps \in \mathcal A$ we decompose the
  energy as $E_\eps = E_\eps^{(1)} + E_\eps^{(2)}$, where
  \begin{align}     \label{Eeps12} %
    \begin{aligned}
      \eps^{-4/3} E_\eps^{(1)} (u_\eps) %
      & = \frac{\lambda^2 \ell^3}{2\kappa^2} - \frac
      {2\lambda}{\kappa^2} \int_{\TT_\ell} d \mu_\eps(x) + 2
      \int_{\TT_\ell} \int_{\TT_\ell} G_\rho(x-y) \, d \mu_\eps(x) \,
      d \mu_\eps(y), \\
      \eps^{-4/3} E_\eps^{(2)} (u_\eps) %
      & = \eps^{-1/3} \sigma \int_{\TT_\ell} |\nabla \chi_\eps| \, dx
      + 2 \eps^{-4/3} \int_{\TT_\ell} \int_{\TT_\ell}
      H_\rho(x-y) \chi_\eps(x) \chi_\eps(y) \, dx \, dy,
    \end{aligned}
  \end{align}
  where $\chi_\eps$ is as in \eqref{chiu} with $u$ replaced with
  $u_\eps$.  The term $E_\eps^{(1)}$ is continuous with respect to the
  weak convergence of measures, hence
  \begin{align}
    \int_{\TT_\ell} \int_{\TT_\ell} G_\rho(x-y) \, d \mu_\eps(x) \,
    d \mu_\eps(y) \to \int_{\TT_\ell} \int_{\TT_\ell} G_\rho(x-y) \, d \mu(x) \,
    d \mu(y) \qquad \text{as} \ \eps \to 0.
  \end{align}
  The proof of the lower bound for $E_\eps^{(2)}$ follows with similar
  arguments as in \cite{kmn:cmp16}. After the rescaling in
  \eqref{chitil}, one can write
  \begin{align}
    \eps^{-4/3} E_\eps^{(2)} (u_\eps) = \left( {\eps^{1/3} \sigma^{5/3}
    \over 4^{2/3}} \right) \left[ \int_{\TT_{\ell_\eps}} |\nabla
    \tilde \chi_\eps| dx + \frac12 \int_{\TT_{\ell_\eps}} \int_{\TT_{\ell_\eps}}
    \widetilde H_\rho^\eps(x - y) \tilde \chi_\eps(x) \tilde \chi_\eps(y)
    \, dx \, dy \right],
  \end{align}
  where $\tilde \chi_\eps(x) := \chi_\eps(x \ell / \ell_\eps)$ and 
  \begin{align}
    \label{Htil}
    \widetilde H_\rho^\eps(x) := (1 - \eta_\rho(x \ell / \ell_\eps))
    G_\eps(x). 
  \end{align}
  Observe that by \eqref{Gsumeps} and monotonicity of $\eta_\rho(x)$
  in $|x|$ we have
  \begin{align*}
    \widetilde H_\rho^\eps(x) \geq  (1 - \rho) \Gamma_{\rho_0}^\#(x),
  \end{align*}
  where $\Gamma_{\rho_0}^\#(x) := (1 - \eta_{\rho_0}(x)) \Gamma^\#(x)$
  and $\Gamma^\#(x) := \frac 1{4\pi |x|}$ is the restriction of the
  Newton potential on the torus, for any $\rho_0 > 0$ and all $\eps$
  small enough depending only on $\kappa$, $\sigma$ and $\rho_0$. The
  rest of the proof follows exactly as in \cite{kmn:cmp16}.

  Finally, if $\mu \not\in H^{-1}(\TT_\ell)$, then
  $E_0(\mu) = +\infty$ and the upper bound is trivial, while the lower
  bound follows via a contradiction argument from the compactness
  established in Theorem \ref{thm-Ecompact}.
\end{proof}

As a direct consequence of Theorems \ref{thm-Ecompact} and
\ref{teogamma}, we have the following characterization of the
minimizers of the sharp interface energy in the limit $\eps \to 0$.

\begin{corollary}
  \label{c:min}
  Let $(u_\eps) \in \mathcal A$ be minimizers of $E_\eps$. Let
  $\mu_\eps$ be defined in \eqref{def-mueps} and let $v_\eps$ be the
  solution of \eqref{def-veps} respectively. Then as $\eps \to 0$, we
  have
\begin{align} 
  \mu_\eps \rightharpoonup \bar\mu \textrm{ in } \mathcal M(\TT_\ell),
  \qquad v_\eps \rightharpoonup \bar v \textrm{ in } H^1(\TT_\ell), 
\end{align}
where $\bar\mu$ and $\bar v$ are as in Proposition \ref{p:minE0}.
\end{corollary}

We note that for $\lambda \gg \lambda_c$ the minimum energy per unit
volume for minimizers in Proposition \ref{c:min} approaches
asymptotically to that of the unscreened sharp interface energy
studied in \cite{kmn:cmp16}, indicating that the presence of an
additional screening does not affect the limit behavior of the energy
at higher densities than those appearing in \eqref{def-ueps}. We would
thus expect that the same result would still hold for the sharp
interface energy even for $1 + \bar u_\eps = o(1)$ as $\eps \to 0$,
consistently with a recent result for the sharp interface energy
without screening \cite{emmert18}.


We conclude by proving an analog of \cite[Theorem 3.6]{kmn:cmp16}
that provides uniform bounds on the diameter of the connected
components of minimizers of $E_\eps$ as $\eps \to 0$, and convergence
of most of the connected components to minimizers of Gamow's model per
unit mass.

\begin{theorem}[Minimizers: droplet structure] \label{t:diam} 
  For $\lambda > 0$, let $(u_\eps) \in \mathcal A$ be regular
  representatives of minimizers of $E_\eps$, and assume that the sets 
  $\{u_\eps = +1\}$ are non-empty for $\eps$ sufficiently small.
  Let $N_\eps$ be the
  number of connected components of the set
  $\{u_\eps = +1\}$, let $\chi_{\eps,k} \in BV(\R^3; \{0, 1\})$ be the
  characteristic function of the $k$-th connected component of the
  support of the periodic extension of $\{u_\eps = +1\}$ to the whole
  of $\R^3$ modulo translations in $\mathbb Z^3$, and let
  $x_{\eps,k} \in \mathrm{supp}(\chi_{\eps,k})$.  Then there exists
  $\eps_0 > 0$ such that the following properties hold:
  \begin{enumerate}
  \item[i)] There exist constants $C, c > 0$ depending only on
    $\sigma$, $\kappa$, $\lambda$ and $\ell$ such that, for
    all $\eps \leq \eps_0$ we have
  \begin{align}
    \label{ccDueps}
    0 < v_\eps \leq C \qquad \text{and} \qquad
    \int_{\R^3} \chi_{\eps,k} \, dx \geq c\, \eps,
  \end{align}
 where $v_\eps$ solves \eqref{def-veps}. Moreover we have
  \begin{align}
    \label{cCNeps}
    \mathrm{supp} (\chi_{\eps,k}) \subseteq B_{C
    \eps^{1/3}}(x_{\eps,k}) .
  \end{align}
  \item[ii)]  If $\lambda>\lambda_c$, where 
 where $\lambda_c$ is given by  \eqref{lamc}, 
 there exist constants $C, c > 0$ as above such that, for all
  $\eps \leq \eps_0$ we have
  \begin{align}
    \label{cCNepslamlamc}
    c  (\lambda - \lambda_c)
    \eps^{-1/3} \leq N_\eps \leq C 
    (\lambda - \lambda_c) \eps^{-1/3} .
  \end{align}
 Moreover, there exists $\widetilde N_\eps \leq N_\eps$ with
  $\widetilde N_\eps / N_\eps \to 1$ as $\eps \to 0$ and a subsequence
  $\eps_n \to 0$ such that for every $k_n \leq \widetilde N_{\eps_n}$
  the following holds: After possibly relabeling the connected
  components, we have
  \begin{align}
   \tilde \chi_n \to \tilde \chi &&\text{in $L^1(\R^3)$}, 
  \end{align}
  where
  $\tilde \chi_n (x) := \chi_{\eps_n,k_n}(\eps_n^{1/3} (x +
  x_{\eps_n,k_n}))$, and
  $\tilde \chi \in \widetilde{\mathcal A}_\infty$ is a minimizer of
  the right-hand side of \eqref{f*}.
  \end{enumerate}
\end{theorem}

\begin{proof}
  The proof can be obtained as in \cite[Theorem 3.6]{kmn:cmp16}, with
  some simplifications due to the absence of a volume constraint. We
  outline the necessary modifications below. As stated above, the
  constants in the estimates below depend on $\sigma$, $\kappa$,
  $\lambda$ and $\ell$, and may change from line to line.

  For $\tilde u_\eps(x) := u_\eps (\ell x / \ell_\eps)$, we define
  $F \subset \TT_{\ell_\eps}$ to be the set $\{\tilde u_\eps = +1\}$,
  which by our assumption 
  is non-empty for $\eps$ sufficiently small. Then we can write
  $v_\eps(x) = (\sigma / 4)^{2/3} v_F(\ell x / \ell_\eps)$, where
  $v_F(x) := \int_{\TT_{\ell_\eps}} G_\eps(x - y) \chi_F(y) dy$, and
  $G_\eps$ is defined in \eqref{Gsumeps}. The first step in the proof
  is to obtain an $L^\infty$-bound on the potential $v_F$ analogous to
  the one in \cite[Lemma 6.3]{kmn:cmp16}:
  \begin{align}\label{sti}
    0< v_F \le  C \eps^{-2/9}.
  \end{align}
  Observe that by strict positivity of $G_\eps$ we clearly have
  $v_F > 0$. On the other hand, the upper bound follows exactly as in
  \cite[Lemma 6.3]{kmn:cmp16}, due to the fact that
  $G_\eps(x) \leq C / |x|$ for some $C > 0$, since
  \begin{align}
    G_\eps(x) = {\ell \over \ell_\eps} G\left( { \ell x \over
    \ell_\eps} \right) = {1 \over 4 \pi} \sum_{\mathbf n \in \mathbb
    Z^3}  {e^{-\kappa \ell |(x / \ell_\eps) - \mathbf n \ell|} \over
    |x - \mathbf n \ell|},  
  \end{align}
  in view of \eqref{Gsum}. 
  
  Next, we need to estimate the gradient of $v_F$ pointwise in terms
  $v_F$ itself, as in \cite[Lemma 6.5]{kmn:cmp16}, which relies on
  \cite[Eq. (6.15)]{kmn:cmp16}.  It is easy to see that the latter
  estimate still holds in the present setting, with the constants
  depending on $\kappa$ and $\ell$. The proof then follows as in
  \cite{kmn:cmp16}, with a few simplifications due to positivity of
  $G$. Also, since $v_F$ satisfies
  \begin{align}
    -\Delta v_F + \left( { \eps \sigma \over 4} \right)^{2/3} \kappa^2
    v_F = \left( { \sigma \over 4} \right)^{2/3} \chi_F \qquad
    \text{in} \ \TT_{\ell_\eps},
  \end{align}
  by positivity of $v_F$ we have that $v_F$ is subharmonic outside
  $\overline F$. Thus, $v_F$ attains its global maximum in
  $\TT_{\ell_\eps}$ for some $\bar x \in \overline F$, and the analog
  of \cite[Eq. (6.19)]{kmn:cmp16} holds true:
  \begin{align}\label{sto}
    v_F(x) \geq \frac34 v_F(\bar x) - C \qquad \text{for all} \ x \in
    B_r(\bar x),
  \end{align}
  for some $C > 0$ and $r > 0$. 
  
  Proceeding as in \cite[Lemma 6.7 and Proposition 6.2]{kmn:cmp16}, we
  establish a lower density estimate for $F$: Given
  $x_0\in \overline F$ and letting $F_0$ be the connected component of
  $F$ containing $x_0$, we have
   \begin{align}\label{sta}
     |F_0\cap B_r(x_0)|\ge c r^3\qquad \text{for all } r\le C
     \min\big(1,\| v_F\|_\infty^{-1}\big) 
     \le C \eps^{2/9},
  \end{align}
  for some $c,\,C>0$, where the last inequality follows from
  \eqref{sti}.  The assertion in \eqref{ccDueps} then follows as in
  \cite[Theorem 6.9]{kmn:cmp16} from \eqref{sti}, \eqref{sto} and
  \eqref{sta}. The idea of the proof in \cite{kmn:cmp16} is to find a
  suitable competitor $F'$ which is obtained by cutting from $F$ a
  ball of radius independent of $\eps$, centered at the point where
  the potential $v_F$ attains its maximum.  Compared to
  \cite{kmn:cmp16}, the proof here is simpler since we don't have a
  volume constraint, so that we can allow competitors with smaller
  volume than $F$.  Arguing by contradiction, if the maximum of $v_F$
  is large, then necessarily the density of $F$ in the ball has to be
  small, otherwise the energy of $F'$ would be less that the energy of
  $F$. However, this contradicts the density estimate in \eqref{sta}.
%
  Finally, exactly as in \cite[Lemma 6.11]{kmn:cmp16}, the bound on
  the potential and the density estimate \eqref{sta} also imply the
  diameter bound
\begin{align}\label{diam}
   \text{diam}(F_0)\le C,
  \end{align}
for some constant $C>0$, which gives \eqref{cCNeps}.
This concludes the proof of part $i)$.
 
 \smallskip
 
 The proof of part $ii)$ follows as in the proof of \cite[Theorem
 3.6]{kmn:cmp16}, with the exception that the estimate on $N_\eps$ in
 \eqref{cCNepslamlamc} now follows from \eqref{ccDueps} and the fact
 that, recalling Corollary \ref{c:min},
\begin{align}
  \lim_{\eps\to 0} \int_{\TT_\ell} d\mu_\eps = \bar\mu(\TT_\ell) =
  \frac 12 (\lambda-\lambda_c), 
\end{align}
where $\bar \mu$ is as in Proposition \ref{p:minE0}.
\end{proof}

The results obtained in Theorem \ref{t:diam} allow us to establish a
sharp transition from trivial to non-trivial minimizers at the level
of the sharp interface energy near $\lambda = \lambda_c$ for all
$\eps \ll 1$.

\begin{corollary}\label{t:triv}
  There exists $\eps_0 = \eps_0(\sigma,\kappa,\lambda, \ell) > 0$ such
  that if $\lambda_c$ is given by \eqref{lamc}, then:
\begin{enumerate}
\item[i)] For any $\lambda < \lambda_c$ and $\eps<\eps_0$ we have
  that $u= -1$ is the unique minimizer of $E_\eps$ in $\mathcal A$.
\item[ii)] For $\lambda>\lambda_c$ and $\eps < \eps_0$ we have that
  $u= -1$ is not a minimizer of $E_\eps$ in $\mathcal A$.
\end{enumerate}
\end{corollary}

\begin{proof}
  Since the statement in {\it ii)} follows immediately from Corollary
  \ref{c:min}, we only need to demonstrate {\it i)}. The strategy is
  analogous to the one used in the proof of \cite[Proposition
  3.2]{m:cmp10}. For $\lambda < \lambda_c$, let $u_\eps$ be a
  minimizer of $E_\eps$ over $\mathcal A$, and assume, by
  contradiction, that $u_\eps \not= -1$ for a sequence of
  $\eps \to 0$. Let $\chi_{\eps,k}$ be as in Theorem \ref{t:diam}. By
  \eqref{Gsum} we have
  \begin{align}
    E_\eps(u_\eps) \geq {\eps^{4/3} \lambda^2 \ell^3 \over 2 \kappa^2}
    + \sum_{k=1}^{N_\eps} \Bigg( \eps \sigma \int_{\mathbb
    R^3} |\nabla \chi_{\eps,k}| dx
    - {2 \eps^{2/3} \lambda \over \kappa^2} \int_{\mathbb
    R^3} \chi_{\eps,k} dx  \notag \\
    + {1 \over 2 \pi} \int_{\mathbb
    R^3} \int_{\mathbb R^3} {e^{-\kappa |x - y|} \over |x - y|}
    \chi_{\eps,k}(x) \chi_{\eps,k}(y) \, dx \, dy \Biggr).
  \end{align}
  At the same time, since by Theorem \ref{t:diam} the diameter of the
  support of $\chi_{\eps,k}$ is bounded above by $C \eps^{1/3}$, for
  every $\delta > 0$ we have $e^{-\kappa |x - y|} \geq 1 - \delta$ for
  all $\eps$ sufficiently small and all
  $x, y \in \text{supp}(\chi_{\eps,k})$. Introducing
  $\tilde \chi_{\eps,k}(x) := \chi_{\eps,k}(\ell_\eps x / \ell)$ as in
  \eqref{chitil}, we can then write
  \begin{align}
    \label{cutthem}
      E_\eps(u_\eps)
      & \geq {\eps^{4/3} \lambda^2 \ell^3 \over 2
        \kappa^2} - {\eps^{5/3} \sigma \lambda \over 2 \kappa^2}
        \sum_{k=1}^{N_\eps} \int_{\mathbb R^3} \tilde \chi_{\eps,k} dx
        \notag \\
      & + {\eps^{5/3} \sigma^{5/3} \over 4^{2/3}} \sum_{k=1}^{N_\eps}
        \Bigg(  \int_{\mathbb R^3} |\nabla \tilde \chi_{\eps,k}| dx + {1
        - \delta \over 8 \pi} \int_{\mathbb R^3} \int_{\mathbb R^3}
        {\tilde \chi_{\eps,k}(x) \tilde \chi_{\eps,k}(y) \over | x - y
        |} dx \, dy \Bigg) \notag \\
      & \geq  {\eps^{4/3} \lambda^2 \ell^3 \over 2
        \kappa^2} - {\eps^{5/3} \sigma \lambda \over 2 \kappa^2}
        \sum_{k=1}^{N_\eps} \int_{\mathbb R^3} \tilde \chi_{\eps,k} dx
        +  {\eps^{5/3} \sigma^{5/3} (1 - \delta) \over 4^{2/3}}
        \sum_{k=1}^{N_\eps} \widetilde E_\infty(\tilde
        \chi_{\eps,k}). 
  \end{align}
  
    Now we substitute the definitions of $f^*$ and $\lambda_c$ in
    \eqref{f*} and \eqref{lamc}, respectively, into
    \eqref{cutthem}. This yields
    \begin{align}
      E_\eps(u_\eps) \geq  {\eps^{4/3} \lambda^2 \ell^3 \over 2
      \kappa^2} + {\eps^{5/3} \sigma \left( (1 - \delta) \lambda_c  -
      \lambda \right) \over 2 \kappa^2} \sum_{k=1}^{N_\eps}
      \int_{\mathbb R^3} \tilde \chi_{\eps,k} dx. 
    \end{align}
    In particular, for $\lambda < \lambda_c$ one can choose $\delta$
    small enough, so that
    $E_\eps(u_\eps) > \eps^{4/3} \lambda^2 \ell^3 / (2 \kappa^2) =
    E_\eps(-1)$ for all $\eps$ sufficiently small, contradicting
    minimality of $u_\eps$.
\end{proof}

\section{Diffuse interface energy $\EE_\eps$} \label{sec:diffuse}

We now consider the diffuse-interface functional $\EE_\eps$
defined in \eqref{EEOK} in the limit $\eps \to 0$ with, as before, $\bar u_\eps$
given by \eqref{def-ueps} and positive $\sigma, \lambda, \kappa, \ell$
fixed.

Let
\begin{align}
  \label{mueps0}
  d \mu^0_\eps(x) := \dfrac12 \eps^{-2/3} \left(1 +
  u^0_\eps(x)\right) dx,
\end{align}
where
\begin{align}
  \label{ueps0}
  u^0_\eps(x) := 
\begin{cases}
  +1 & {\rm if}\ u_\eps(x)>0,
  \\
  -1 & {\rm if}\ u_\eps(x)\le 0,
\end{cases}
\end{align}
and let $v_\eps^0$ satisfy
\begin{align}\label{vdef0}
  -\Delta v_\eps^0 + \kappa^2 v_\eps^0 =
  \mu^0_\eps \qquad \text{in} \quad \TT_\ell. 
\end{align}
With this notation, we are now in the position to state the main
technical result of this paper.

\begin{theorem}[Equicoercivity and $\Gamma$-convergence of $\E_\eps$]
  \label{teogammabis}
  \noindent For $\lambda > 0$ and $\ell > 0$, let $\E_\eps$ be defined
  by (\ref{EEeps}) with $W$ satisfying the assumptions of
  Sec. \ref{subsec:diffuse}, let $\bar u_\eps$ given by \eqref{def-ueps},
  and let $\sigma$ and $\kappa$ be given by \eqref{Wint} and
  \eqref{kappa}, respectively. Then, as $\eps \to 0$ we have
  \begin{align}
    \label{EEGamma}
    \eps^{-4/3} \E_\eps \stackrel{\Gamma}\rightharpoonup E_0(\mu),
  \end{align}
  where $\mu \in \mathcal{M}^+(\TT_\ell) \cap H^{-1}(\TT_\ell)$.  More
  precisely, we have
  \begin{itemize} \item[i)] Compactness and lower bound: Let
    $(u_\eps) \in \mathcal A_\eps$ be such that
    $\limsup_{\eps \to 0} \| u_\eps \|_{L^\infty(\TT_\ell)} \leq 1$
    and
  \begin{align}\label{ZZO.1} \limsup_{\varepsilon \to 0}
    \eps^{-4/3} \E_\eps(u_\eps) < +\infty.
\end{align}
Then, up to extraction of a subsequence, we have
\begin{align} 
  \mu^0_\eps \rightharpoonup \mu \textrm{ in } \mathcal M(\TT_\ell),
  \qquad v_\eps^0 \rightharpoonup v \textrm{ in } H^1(\TT_\ell), 
\end{align}
as $\eps \to 0$, where
$\mu \in \mathcal{M}^+(\TT_\ell) \cap H^{-1}(\TT_\ell)$ and
$v \in H^1(\TT_\ell)$ satisfy
\begin{align} \label{PDEE}
  -\Delta v + \kappa^2 v = \mu
  \qquad \text{in} \quad \TT_\ell.
\end{align}
Moreover, we have
\begin{align}
  \label{EEliminf}
  \liminf_{\varepsilon \to 0}  \eps^{-4/3}
  \E_\eps(u_\eps) \geq E_0(\mu).
\end{align}
\item[ii)] Upper bound: Given
  $\mu \in \mathcal{M}^+(\TT_\ell) \cap H^{-1}(\TT_\ell)$ and
  $v \in H^1(\TT_\ell)$ solving \eqref{PDEE}, there exist
  $(u_\eps) \in \mathcal A_\eps$ such that for the corresponding
  $\mu^0_\eps$, $v_\eps^0$ as in \eqref{mueps0} and (\ref{vdef0}) we
  have
\begin{align} 
  \mu^0_\eps \rightharpoonup \mu \text{ in } \mathcal M(\TT_\ell),
  \qquad v_\eps^0 &\rightharpoonup v \textrm{ in } H^1(\TT_\ell), 
\end{align}
as $\eps \to 0$, and
\begin{align} \label{EElimsup} %
  \limsup_{\varepsilon \to 0} \eps^{-4/3} \E_\eps(u_\eps) \leq E_0(\mu).
\end{align}
\end{itemize}
\end{theorem}

\begin{proof}
  As in \cite{m:cmp10}, the basic strategy is to relate the
  minimization problem for $\mathcal E_\eps$ to that for $E_\eps$ and
  apply the results in Theorems \ref{thm-Ecompact} and \ref{teogamma}.
  The proof relies on the fact, first observed in \cite{m:cmp10}, that
  the energy $\EE_\eps$ is asymptotically equivalent to $E_\eps$ in
  the following sense: For any $\delta > 0$ and
  $u_\eps \in \mathcal A_\eps$ satisfying some mild technical
  conditions (see below) there is $\tilde u_\eps \in \mathcal A$ such
  that
  \begin{align}
    \label{EElb}
    E_\eps[\tilde u_\eps] \leq (1+\delta) \EE_\eps(u_\eps),
  \end{align}
  for all $\eps \ll 1$, and, conversely, for any
  $\tilde u_\eps \in \mathcal A$, again, satisfying some mild
  technical conditions, there is $u_\eps \in \mathcal A_\eps$ such
  that
  \begin{align}
    \label{EEub}
    \EE_\eps[u_\eps] \leq (1+\delta) E_\eps(\tilde u_\eps),
  \end{align}
  for all $\eps \ll 1$. The proof then proceeds exactly as in the
  two-dimensional case \cite[Theorem 1]{gms:arma13}, with
  modifications appropriate to three space dimensions. We outline the
  key differences below.

  For \eqref{EElb} to hold, we need to verify the assumptions of
  \cite[Proposition 4.2]{m:cmp10}, which are equivalent to checking
  that $\| u_\eps \|_\infty \to 1$, $\mathcal E_\eps (u_\eps) \to 0$
  and $\| v_\eps \|_\infty \to 0$ as $\eps \to 0$, where
  $v_\eps(x) := \int_{\TT_\ell} G_0(x - y) (u_\eps(y) - \bar u_\eps)
  dy$. The first and second conditions are clearly satisfied by the
  assumptions of the theorem. To check the third condition, we note
  that the non-local part of the energy may be written in terms of
  $v_\eps$ as
  \begin{align}
    \label{nlocgradveps}
    \frac12 \int_{\TT_\ell} \int_{\TT_\ell}   (u_\eps(x) - \bar u_\eps) G_0(x
    - y) (u_\eps(y) - \bar u_\eps) \, dx \, dy =     \frac12 \int_{\TT_\ell}
    |\nabla v_\eps|^2 dx.
  \end{align}
  Since $\int_{\TT_\ell} v_\eps(x) \, dx = 0$ we have by Poincar\'e's
  inequality that the right-hand side of \eqref{nlocgradveps} is
  bounded below by a multiple of $\| v_\eps\|_2^2$. In turn, the
  latter is bounded below by a multiple of $\| v_\eps \|_\infty^5$, in
  view of the fact that by elliptic regularity we have
  $\|\nabla v_\eps \|_\infty \leq C$ for some $C > 0$ depending only
  on $\ell$, for all $\eps \ll 1$. Therefore, from
  $\mathcal E_\eps(u_\eps) \to 0$ we also obtain that
  $\|v_\eps\|_\infty \to 0$ as $\eps \to 0$. Thus, by \eqref{EElb}
  $\tilde u_\eps$ satisfies the assumptions of Theorem
  \ref{thm-Ecompact}, and so there exists
  $\mu \in \mathcal M(\TT_\ell)$ such that, upon extraction of
  subsequences, $\mu_\eps \rightharpoonup \mu$ in
  $\mathcal M(\TT_\ell)$, where the measure $\mu_\eps$ defined by
  \eqref{def-mueps} with $u_\eps$ replaced by $\tilde u_\eps$. For
  those subsequences, Theorem \ref{teogamma} holds true for $\mu_\eps$
  as well.
  
  Now, from the construction of $\tilde u_\eps$ in the proof of
  \cite[Lemma 4.1]{m:cmp10} we know that
  $\tilde u_\eps(x) = u_\eps^0(x)$ for all $x \in \TT_\ell$ such that
  $|u_\eps(x)| > 1 - \delta^2$. Hence from the bound on
  $\mathcal E_\eps(u_\eps)$ and the assumptions on $W$ we get that
  $\| \tilde u_\eps - u_\eps^0 \|_1 \leq C \eps^{4/3} \delta^{-4}$ for
  some $C > 0$ and all $\eps \ll 1$. This implies that
  $\mu_\eps^0 \rightharpoonup \mu$ in $\mathcal M(\TT_\ell)$ as well
  $\eps \to 0$. Together with the conclusions of Theorems
  \ref{thm-Ecompact} and \ref{teogamma}, this gives the compactness
  and the lower bound statement of Theorem \ref{teogammabis}, in view
  of arbitrariness of $\delta$.

  For \eqref{EEub} to hold, we need to verify the assumptions of
  \cite[Proposition 4.3]{m:cmp10} on $\tilde u_\eps \in \mathcal A$,
  namely, that the connected components of the support of
  $\{ \tilde u_\eps = +1\}$ are smooth and at least $\eps^\alpha$
  apart for some $\alpha \in [0,1)$, have boundaries whose curvature
  is bounded by $\eps^{-\alpha}$, and that
  $\| \tilde v_\eps \|_\infty \to 0$ as $\eps \to 0$, where
  $\tilde v_\eps(x) := \int_{\TT_\ell} G(x - y) (\tilde u_\eps(y) -
  \bar u_\eps) dy$. Clearly the first two assumptions hold true for
  the recovery sequence in the proof of Theorem \ref{teogamma} with
  any $\alpha \in (\frac13, 1)$, provided that $\eps \ll 1$. The third
  assumption is satisfied for all $\eps \ll 1$, in view of the fact
  that the non-local part of the sharp interface energy can be written
  as
  \begin{align}
    \frac12 \int_{\TT_\ell} \int_{\TT_\ell}   (\tilde u_\eps(x) - \bar u_\eps) G_0(x
    - y) (\tilde u_\eps(y) - \bar u_\eps) \, dx \, dy =     \frac12 \int_{\TT_\ell}
    \left( |\nabla \tilde v_\eps|^2 + \kappa^2 \tilde v_\eps^2 \right) dx,
  \end{align}
  and the desired estimate follows from $E_\eps(\tilde u_\eps) \to 0$
  just like in the case of the diffuse interface energy. Thus, the
  proof of the upper bound is concluded by taking the functions
  $u_\eps$ appearing in \eqref{EEub}, associated with the recovery
  sequence $(\tilde u_\eps)$ from Theorem \ref{teogamma}, once again,
  in view of arbitrariness of $\delta$.
\end{proof}

Similarly to the sharp interface energy, as a direct consequence of
Theorem \ref{teogammabis} we have the following characterization of
the minimizers of the diffuse interface energy in the limit
$\eps \to 0$.

\begin{corollary}
  \label{c:diffmin}
  Under the assumptions of Theorem \ref{teogammabis}, let
  $(u_\eps) \in \mathcal A_\eps$ be minimizers of $\mathcal
  E_\eps$. Let $\mu^0_\eps$ be defined in \eqref{mueps0} and
  $v_\eps^0$ be the solution of \eqref{vdef0}, respectively. Then as
  $\eps \to 0$, we have
\begin{align} 
  \mu^0_\eps \rightharpoonup \bar\mu \textrm{ in } \mathcal M(\TT_\ell),
  \qquad v_\eps^0 \rightharpoonup \bar v \textrm{ in } H^1(\TT_\ell), 
\end{align}
where $\bar\mu$ and $\bar v$ are as in Proposition \ref{p:minE0}.
\end{corollary}

We emphasize that the limit behavior of the minimal energy obtained in
\eqref{c:diffmin} {\em differs} from that of the unscreened sharp
interface energy one would naively associate with $\mathcal
E_\eps$. In particular, the minimal energy exhibits a threshold
behavior, contrary to that of the minimizers of the unscreened sharp
interface energy studied in \cite{kmn:cmp16,emmert18}.

\smallskip

\begin{proof}[Proof of Theorems \ref{t:main} and \ref{t:f*}] The
  statement of Theorem \ref{t:main} is simply the restatement of
  Corollary \ref{c:diffmin} that does not specify the precise values
  of the constants appearing there. In turn, the statement of Theorem
  \ref{t:f*} uses the explicit values of
  $\sigma = {2 \sqrt{2} \over 2}$ and $\kappa = {1 \over \sqrt{2}}$
  for \eqref{EEOK}, together with the bounds on $f^*$ obtained in
  \eqref{f*B} and \eqref{eq:f*lb}.
\end{proof}

\paragraph*{Acknowledgements.}

HK was supported by DFG via grant \#392124319. CBM was supported, in
part, by NSF via grants DMS-1313687 and DMS-1614948.  MN was partially
supported by INDAM-GNAMPA, and by the University of Pisa Project PRA
2017 ``Problemi di ottimizzazione e di evoluzione in ambito
variazionale''.

\bibliography{../nonlin,../mura,../stat}
\bibliographystyle{plain}

\end{document}